\documentclass[12pt,a4paper]{article}
\usepackage[english]{babel}
\usepackage{bm}
\usepackage{amsmath,amsthm,amssymb}
\usepackage{latexsym}
\usepackage{hyperref}
\usepackage{enumerate}
\usepackage{xcolor}
\usepackage{courier}        
\usepackage{graphicx}

\newcommand*{\Scale}[2][4]{\scalebox{#1}{$#2$}}%

\newtheorem{thm}{Theorem}

\newtheorem{lemm}{Lemma}[section]

\newtheorem{defi}[lemm]{Definition}
\newtheorem{prop}[lemm]{Proposition}

\newtheorem{remm}[lemm]{Remmark}

\newcommand{\Ccal}{\mathcal{C}}

\newcommand{\PP}{\mathbb{P}}

\newcommand{\Rd}{\mathbb{R}^d}
\newcommand{\R}{\mathbb{R}}
\newcommand{\N}{\mathbb{N}}
\newcommand{\T}{\mathbb{T}}
\newcommand{\D}{\mathcal{D}}

\title{Convergence Rate for Moderate Interaction  Particles and Application to Mean Field Games}
\author{Josu\'e Knorst, Christian Olivera and Alexandre B. de Souza}
\date{}

\begin{document}
\maketitle

\vspace{0.3cm} \noindent {\bf MSC2010 subject classification:} 49N90,  60H30, 60K35.

\begin{abstract}
We study  two interacting particle systems, both modeled as a system of  $N$ stochastic differential equations driven by Brownian motions with singular kernels and  moderate interaction. We  show  a quantitative result where the convergence rate depends on the moderate scaling parameter, the regularity of the solution of the limit equation and the dimension. Our approach is based on the techniques   of stochastic calculus, some  
 properties of  Besov and Triebel-Lizorkin space,  and the  semigroup approach introduced in \cite{FlandoliLeimbachOlivera}. \color{red}{New techniques are presented to address the difficulty arising from the nonlinear term.}
\end{abstract}

\section{Introduction and  results.}

\subsection{Fokker-Planck equations with singular kernel.}

In the first part of this paper, we consider  the stochastic approximation of nonlinear Fokker-Planck Partial Differential Equations (PDEs) of the form
\begin{equation}\label{eq:PDE}
\begin{cases}
&\partial_t p(t,x) = \Delta p(t,x) - \nabla \cdot \big(p(t,x) ~  K \ast_{x} p (t,x) \big),\quad t>0,~x\in\T^d,\\
&p(0,x) = p_0(x), 
\end{cases}
\end{equation}
by means of moderately interacting particle systems. The main interest here is kernels $K$ with a singularity at the origin. {\color{red}{Equation (\ref{eq:PDE}) is   attracting the attention of
a large scientific community, with a large amount of publications in the literature. Accordingly the nature of the interaction kernel can lead to diverse phenomena, including  classical models such as the $2d$ Navier-Stokes equation, which in vorticity form can be written as in \eqref{eq:PDE} with the Biot-Savart kernel,  and the parabolic-elliptic Keller-Segel PDE in any dimension $d\geq 1$, which models the phenomenon of chemotaxis. See   \cite{Biler}, \cite{BLanchet}, \cite{Bogachev2},  \cite{Burger2} and  \cite{Meajda} for further details of the mentioned models.}}

Here, we establish a rigorous approximation  of the 
equation \eqref{eq:PDE} from an interacting 
stochastic particle system. More precisely, we  consider the following $N$-particle dynamics described, for each $N\in\mathbb{N}$, by the system coupled stochastic differential equations in $\mathbb{T}^{d}$:

\begin{equation}
dX_{t}^{i,N}=  \frac{1}{N}\sum_{k=1}^{N} (K\ast V^{N})(X_{t}^{i,N} -X_{t}^{k,N}) \; dt + \sqrt{2} \; dW_{t}^{i} \label{itoass}
\end{equation}

  \noindent $\{W_{t}^{i}, \; i\in  \N \}$ is a family of independent standard Brownian motions on $\T^d$ defined on a filtered probability space $\left(\Omega,\mathcal{F},\mathcal{F}_{t},\mathbb{P}\right)  $. The \emph{interaction potential} $V^N: \T^d \rightarrow \R_+$ is differentiable and will be specified later.

We now define the \emph{empirical process}
\[S_{t}^{N}:=\frac{1}{N} \sum_{i=1}^{N}\delta_{X_{t}^{i,N}} \] which is the (scalar) measure-valued
process associated to the {\color{red}$\T^d$-valued} processes $\{t\mapsto X_{t}^{i,N}\}_{i=1,...,N}$. Above,
$\delta_{a}$ is the delta Dirac measure concentrated at {\color{red}$a \in \T^d$}. For any test function $\phi: \T^d \to \R$, we use the standard notation \[\langle S_t^N,\phi \rangle := \frac{1}{N}\sum_{i=1}^N \phi(X_t^{i,N}). \] In what follows $\mu \ast \nu$ stands for the standard convolution between two measures $\mu$ and $\nu$. Our interest lies
in investigating the dynamical process
$t\mapsto S_{t}^{N}$ in the limit $N\rightarrow\infty$.

If one takes in \eqref{itoass} an interaction potential $V^N \approx V$ which does not depend on the parameter $N$, one expects a non-local term of
the form {\color{red}$V\ast p$} in the limiting PDE; this is the \emph{mean field regime}. In order
to get a \emph{local} term, it is necessary to rescale the potential $V$. This is why we assume that the sequence $V^{N}$  concentrates as
$N\rightarrow\infty$: following the proposal of  K. Oelschl\"{a}ger  \cite{Oelschlager85}, see also \cite{JourdainMeleard} and \cite{Meleard}, we suppose that, for any $x \in \mathbb \T^d$,

\begin{equation}\label{eq:VN}
V^{N}(x)=N^{\beta}V(N^{\frac{\beta}{d}}x), \qquad \text{for some } \beta \in [0,1],
\end{equation}

\noindent where $V:\T^d \to \R_+$ is a smooth probability density. The case $\beta = 0$ is the \emph{mean field} one (\textit{i.e.}~long range interaction), the
case $\beta = 1$ corresponds to \emph{local} (\textit{i.e.}~nearest neighbor) interactions, while the
case $0 < \beta < 1$ corresponds to an intermediate regime, called \emph{moderate}
by  \cite{Oelschlager85}. Moderately interacting many-particle systems are essentially characterized by the feature that the range of the different types of interactions between the particles is
both large in comparison to the typical distance between neighboring particles and
small in comparison to the spatial size of the whole system.
 A further discussion of this limiting concept can be found in \cite{Oelschlager87}.
About more advances in particle systems, moderate and mean-field limits see, for instance, \cite{BURGER}, \cite{Ansgar}, \cite{Chen}, \cite{correa},  \cite{FlandoliLeimbachOlivera}, \cite{Ansgar2},\cite{Pisa} and \cite{Steve}.

In this paper, we quantify  the distance (of the trajectory ) between the empirical measure and the solution of (\ref{eq:PDE}) ; see Theorem  \ref{th:rate2}.  The proof is based in semigroup approach introduced in \cite{FlandoliLeimbachOlivera}, with applications to a large class of PDE,  see    \cite{FlandoliLeocata},  \cite{FlandoliLeocataRicci0}, \cite{Leocata}, \cite{FlandoliOliveraSimon} \cite{Luo},  \cite{Pisa} and \cite{Simon}. This approach  was greatly improved in \cite{Pisa} and \cite{ORT} for singularly
interacting particle systems,  providing quantitative convergence rates.
The novelty in our analysis, compared to \cite{Pisa} and \cite{ORT}, lies in the analysis of the convergence (of the trajectory) in Besov and Triebel-Lizorkin spaces. {\color{red} Moreover we quantify the convergence of particle systems with H\"older drift and apply this for the convergence of the mean field games. This  requires some new technical ingredients. The approach to prove these results is to write the mild equation satisfied by $p^N_t=V^N \ast S^N_t$ and then write the difference $p-p^N$, for the mild solution $p$ of \eqref{eq:PDE}, then we introduced  appropriate stopping times
and use some the semigroup properties in Besov and Triebel-Lizorkin space in order to get our results}.

\subsection{Fokker-Planck equations with H\"older drift term.}

In the second part of this work, we consider  the stochastic approximation of nonlinear Fokker-Planck Partial Differential Equations (PDEs) of the form
\begin{equation}\label{eq:PDE2}
\begin{cases}
&\partial_t p(t,x) = \Delta p(t,x) - {\color{red}\nabla \cdot \big( p(t,x)(\alpha(t,x) +  b(x,p)) \big)},\quad t>0,~x\in\R^d,\\
&p(0,x) = p_0(x), 
\end{cases}
\end{equation}
The main motivation  to consider this class of  PDE is to approximate 
the  Mean Field Game (MFG) system of PDEs by stochastic particle systems, see section \ref{MFG}.  

We  consider the following $N$-particle dynamics described, for each $N\in\mathbb{N}$, by the system of coupled stochastic differential equations in $\mathbb{R}^{d}$:

\begin{align}\label{SDE}
dX_t^{i,N}= & \left(\alpha^{i,N}(t, X^{i,N}_t)+b\left(X_t^{i,N}, \frac{1}{N} \sum_{j=1}^N V^N\left(X_t^{i,N}-X_t^{j,N}\right)\right)\right) dt \nonumber \\
& + d W_t^{i,N}, \quad t \in[0, T], i=1,...,N,
\end{align}

\noindent where  $\alpha^{i,N}, \, i=1,...,N$  and $b$  are deterministic functions and $W^{1,N},..., W^{N,N}$ are independent d-dimensional Wiener processes defined on a filtered probability space $\left(\Omega,\mathcal{F},(\mathcal{F}_{t})_{t\geq 0},\PP\right)$.

Our aim is the study of the asymptotic as $N\rightarrow\infty$ of  the time evolution of the whole particle system. Therefore, we investigate the empirical processes

\begin{eqnarray}\label{empp}
S_{t}^{N}:=\frac{1}{N}\sum_{k=1}^{N}\delta_{X_{t}^{k,N}},
\end{eqnarray}

\noindent where $\delta_{a}$  denotes the
Dirac measure at {\color{red}$a \in \mathbb{R}^d$}. The measure $S_{t}^{N}$ determines the distribution of the positions in the $N$-th system. Our contribution is to obtain the convergence rate
of the empirical measure distribution and the limit model  (\ref{eq:PDE2}). The proof of the main theorem, see Theorem \ref{mainthm}, is based on the semigroup mentioned in the previous session.

\subsection{Notations.}

\begin{itemize}

\item We use $\D$ to denote either $\T^d$ or $\R^d$.

\item Whenever we consider a function on $\T^d$, we associate it with its periodic extension, which is a function on $\R^d$.

 \item In this paper, $q\ge 1$ and $(e^{t\Delta })_{t\geq 0 }$ denotes the semigroup of the heat operator on $\D$. 
 That is, for $f \in {L}^q(\D)$, 
\begin{equation*}
e^{t \Delta}f (x)  
 = g_{2t}^{(\D)} \ast_{\D} f(x),
\end{equation*}
where $\ast_{\D}$ denotes the convolution on $\D$ and for any $t>0$, $g_{t}^{(\D)}$ is heat kernel given by
\begin{equation*}
g_{t}^{(\D)}(x) =
\begin{cases}
&\displaystyle\frac{1}{(2\pi t)^{d/2}} \sum_{k\in \mathbb{Z}^{d}} e^{-\frac{| x-k |^{2}}{2t}} ~\mbox{on } \T^d,\\
&\displaystyle\frac{1}{(2\pi t)^{d/2}} e^{-\frac{|x|^{2}}{2t}}~\mbox{on } \R^d.
\end{cases}
\end{equation*}
{\color{red}Also, we recall for $\lambda \in \mathbb{R}$, (see e.g. \cite{Pisa}, p.11)

\begin{equation} \label{potential bessel operator estimate}
 \|(I-\Delta)^{\frac{\lambda}{2}} e^{t\Delta}\|_{L^q(\D)\to L^q(\D)} \leq C t^{-\frac{\lambda}{2}}.
\end{equation}}

\item If $u$ is a function or stochastic process defined on $[0,T]\times \D$, we will most of the time use the notation $u_{t}$ to denote the mapping  $x\mapsto u(t,x)$.

\item Depending on the context, the brackets $\langle\cdot , \cdot \rangle$ will denote either the scalar product in some $L^2$ space or the duality bracket between a measure and a function.

\item For $\mathcal{X}$ some normed vector space, the space $\Ccal([0,T];\mathcal{X})$ of continuous functions from the time interval $[0,T]$ with values in $\mathcal{X}$ is classically endowed with the norm
\begin{align*}
\|f\|_{T,\mathcal{X}} = \sup_{s\in [0,T]} \|f_{s}\|_{\mathcal{X}}.
\end{align*}

\end{itemize}

\paragraph{Bessel spaces.} We now briefly introduce \emph{Bessel potential spaces} on $\R^d$.

\begin{itemize}
\item For any  $\lambda\in \R$ and $q\geq1$,  we denote by
$H^{\lambda}_{q}(\mathbb{R}^{d})$ the space
\[H_{q}^{\lambda}(\mathbb{R}^{d}):= \Big\{ u \text{ tempered distribution; }  \, \mathcal{F}^{-1}\Big( \big(1+|\, \cdot\, |^{2}\big)^{\frac{\lambda}{2}}\; \mathcal{F} u(\cdot) \Big) \in  L^q(\mathbb R^d)\Big\},  \]
where $\mathcal Fu$ denotes the \emph{Fourier transform} of $u$. This space is endowed with the norm
 \begin{equation*}
 \| u \|_{\lambda,q} = \Big\| \mathcal{F}^{-1}\big((1+|\cdot|^{2})^{\frac{\lambda}{2}
}\; \mathcal{F} u(\cdot) \big) \Big\|_{L^q(\R^d)}. 
 \end{equation*}
The space $H_{q}^\lambda(\R^d)$ is 
associated to the Bessel potential operator {\color{red}$(I-\Delta)^\frac{\lambda}{2}$} defined as (see e.g. \cite[p.180]{Triebel} for more details on this operator):
\begin{align*}%
(I-\Delta)^\frac{\lambda}{2} u := \mathcal{F}^{-1}\left((1+|\cdot|^2)^{\frac{\lambda}{2}} \mathcal{F}u \right) .
\end{align*}

\end{itemize}
{\color{red}We recall the following Sobolev embedding, see \cite{Triebel} p.203: for $\lambda > \frac{d}{q}$ and $u \in H^{\lambda}_q$
\begin{align} \label{bessel holder emb}
\left\|u\right\|_{\lambda -\frac{d}{q}} \leq C \left\|u\right\|_{\lambda,q}, 
	\end{align}
where $\left\|\cdot\right\|_{\gamma}$ denotes the norm of a H\"older continuous function.}

{\color{red}Also by definition of Bessel potential operator and inequality $(\ref{potential bessel operator estimate})$, it holds for any $u \in L^q$ and $\lambda \in \mathbb{R}$, 
\begin{equation}\label{Semi}
 \|\nabla e^{t\Delta} u\|_{\lambda,q} \le C t^{-\frac{1+\lambda}{2}}\|u\|_{L^q}.
\end{equation}}

\paragraph{ Triebel-Lizorkin and Besov Space.} 

Let us first define a dyadic partition of unity as follows: we consider two 
$C_{0}^{\infty}(\R^{d})$-functions $\chi$ 
 and $\varphi$  which take values in $[0,1]$ and satisfy the following: there exists $\Lambda \in (1,\sqrt 2)$ such that
\[
\mathrm{Supp}\; \chi=\left\{   |\xi |\leq\Lambda  \right\} \qquad \text{and} \qquad \mathrm{Supp}  \ \varphi=\left\{  {\Lambda}^{-1} \leq|\xi|\leq\Lambda  \right\}.
\]
Moreover, with the following notations,  
\[
\varphi_{-1}(\xi):=\chi(\xi), \qquad
\varphi_{i}(\xi):=\varphi(2^{-i}\xi), \ \text{for any } i\geq 0.
\]
The sequence $\{\varphi_i\}$ satisfies
\begin{align*}
& \mathrm{Supp} \ \varphi_{i}\cap \mathrm{Supp} \ \varphi_{j}=\emptyset	 \quad \text{ if } \ |i-j|>1,  \\
& \sum_{i\geq-1} \varphi_{i}(\xi)=1, \quad \text{ for any } 
\xi \in\R^{d}.
\end{align*}

 Fix a dyadic partition of unity $\{ \varphi_{i} \}$ with its inverse Fourier transforms 
$  \{  \check{\varphi}_{i}   \}$. 
For $u\in \mathcal{S}'(\R^{d})$, the nonhomogeneous dyadic  blocks are defined  as
\[
{\color{red}\varphi_i(D)\equiv 0,} \quad \text{ if } i< -1, \qquad \text{and} \qquad 
{\color{red}\varphi_i(D)u=\check{\varphi}_{i}\ast u,} \quad \text{ if  }  i\geq -1.
\]
The partial sum of dyadic blocks is defined as a nonhomogeneous low frequency cut-off
operator:
{\color{red}\[
S_{j}:=\sum_{i\leq j-1} \varphi_i(D). 
\] }

Having this smooth resolution of unity, we are able 
to introduce the Triebel-Lizorkin and Besov spaces.

\begin{defi}  Let $s\in \R$ and $1< r <\infty$. 
 
\begin{enumerate}

\item  If $1< q< \infty$, then

\[
F_{q,r}^{s}= \left\{  f\in \mathcal{S}^{\prime}(\R^{d}) : \    \big\| \| (2^{js} {\color{red}\varphi_{j}(D)}f)_{j \in \mathbb{Z} } \|_{l^r(\mathbb{Z})}  \big\|_{\mathit{L}^{q}(\R^{d})}< \infty \right\}
\]

\item If $1< q < \infty$, then

\[
B_{q,r}^{s}= \left\{  f\in \mathcal{S}^{\prime}(\R^{d}) : \  \big\| (2^{js}\| {\color{red}\varphi_{j}(D)}f \|_{L^q} )_{j \in \mathbb{Z} } \big\|_{\mathit{l}^{r}(\mathbb{Z})}< \infty \right\}
\]
\end{enumerate}

\end{defi}

The spaces $F_{q,r}^{s}$  and $B_{q,r}^{s}$ are independent of the chosen dyadic partition, see \cite{Triebel}. 
In a similar way, we can define the  Triebel-Lizorkin and Besov space  on $\T^d$, see \cite{SchmeisserTriebel}. {\color{red} Our first result comprehends a convergence which holds in both of these spaces, thus in what follows we denote by $E_{q,r}^{s}$ either $F_{q,r}^{s}$ or $B_{q,r}^{s}$. Noticing that the duals of these spaces are $F_{q^\prime,r^\prime}^{-s}$ and $B_{q^\prime,r^\prime}^{-s}$, where $q^\prime, r^\prime$ are the conjugate indices of $q,r$, (see \cite{Sawano} p.257) we also use $E^{-s}_{q^\prime,r^\prime}$ to denote either of them.
In this notation, we have for $u \in E^{\lambda}_{q,r}$ (see \cite{Sawano} p.644 Theorem 5.29),
\begin{equation} \label{grad heat besov-triebel}
 \|\nabla e^{t\Delta} u\|_{E^\lambda_{q,r}} \le C t^{-\frac{1}{2}}\|u\|_{E^\lambda_{q,r}}.
\end{equation}}
{\color{red} We have  the  Sobolev embedding, see \cite{Triebel} p.203: for $\lambda > \frac{d}{q}$, $r \in (0,\infty)$ and $u \in E^{\lambda}_{q,r}$
\begin{align} \label{besov/triebel holder emb}
	\left\|u\right\|_{\lambda -\frac{d}{q}} \leq C \left\|u\right\|_{E^{\lambda}_{q,r}}, 
\end{align}
where $\left\|\cdot\right\|_{\gamma}$ denotes the norm of a H\"older continuous function.} {\color{red}Finally we recall the  following  Sobolev embedding, (see \cite{Triebel}, p.173, inclusions (4a) and (4b), and p.180, inclusion (9)): for $\lambda ,\delta >0$, $q',r'>1$ and $u \in E^{-\lambda}_{q',r'}$
	\begin{align}  \label{besov/triebel bessel emb}
	\left\|u\right\|_{E^{-\lambda}_{q',r'}} \leq C \left\|u\right\|_{H^{-\lambda+\delta}_{q'}}. 
	\end{align}
	}

\subsection{First main result.}

Solutions to \eqref{eq:PDE} will be understood in the following mild sense:
\begin{defi}\label{def:defMild}
	Given $K\in L^{q^\prime}(\T^{d}) $, $p_0 \in  L^{q}(\T^{d})$ with $1/q+ 1/{q\prime}=1$ and $T>0$,
	a function $p$ on {\color{red}$[0,T] \times \T^d$} is said to be a mild solution to (\ref{eq:PDE}) on $[0,T]$ if
	\begin{enumerate}
		\item $p\in C([0,T]; L^{q}(\T^{d}) ) $; 
		\item $p$ satisfies the integral equation
		\begin{equation}
			\label{eq:mildKS}
			p_{t} =  e^{t\Delta} p_0 -  \int_0^t \nabla \cdot ( e^{(t-s)\Delta }  (p_{s}\,  K \ast p_{s}) )\, ds, \quad 0 \leq t \leq T.
		\end{equation}
	\end{enumerate}
	A function $p$ on $[0,\infty) \times \T^d$ is said to be a global mild solution to \eqref{eq:PDE} if it is a mild solution to \eqref{eq:PDE} on $[0,T]$ for all $T>0$.
\end{defi}

Our first main result is the following claim, whose proof is detailed in Section \ref{subsec:proofRate2}. {\color{red} Recall that $p^N_t=V^N \ast S^N_t$, $t \in[0,T]$.}

\begin{thm}\label{th:rate2}  
	{\color{red}Assume that $K \in L^{q^\prime}(\mathbb{T}^d)$ and there exists $C_K >0$ such that
		\[ \|K \ast f\|_{E^{\lambda}_{q,r}} \le C_K \|f\|_{L^q}    \] \noindent for any $f\in L^{q}(\T^{d}) $ with $q>d/\lambda $ and $\lambda >0$. Let $T_{\max}$ be the maximal existence time for \eqref{eq:PDE}, fix $T\in(0,T_{\max})$ and
		take $\delta>0$ such that $\rho >0$, where
		\begin{align}\label{eq:def_rho}
			\rho =  \min \left\{\frac{\beta}{d}\Big(\lambda-\frac{d}{q}\Big),\, \frac{1}{2} - \frac{\beta}{2}\Big(1 +2 \Big(\frac{1}{2}-\frac{1}{q}\Big)\vee0\Big) \right\}-\delta,
		\end{align}
		
		In addition, let the dynamics of the particle system be given by \eqref{itoass} and assume 
		\begin{align}
			{\color{red} \lim_{N \to \infty} N^{\rho} \, \| p^N_{0}-p_{0} \|_{L^{q}(\T^{d})}=0, \,\, \mathbb{P}\,\text{-a.s}.} 
		\end{align}

		and

		\begin{align}\label{initial}
			\lim_{N \to \infty} N^{\rho}\,  \| S^N_{0}-p_{0} \|_{E_{q^\prime,r^\prime}^{-\lambda}}=0, \,\, \mathbb{P}\,\text{-a.s},
		\end{align}
		with \( q' \) and \( r' \) as the conjugate indices of \( q \) and \( r \), respectively.

		Then, we have

		\begin{align}\label{rate_moll}
			\lim_{N \to \infty} N^{\rho}\sup_{t \in [0,T]} \| p^N_{t}-p_{t} \|_{L^{q}(\T^{d})}=0, \,\, \mathbb{P}\,\text{-a.s}.
		\end{align}

		and

		\begin{align}\label{rate_true}
			\lim_{N \to \infty}N^{\rho}\sup_{t \in [0,T]} \| S^N_{t}-p_{t} \|_{E_{q^\prime,r^\prime}^{-\lambda}}=0, \,\, \mathbb{P}\,\text{-a.s}.
	\end{align}}

\end{thm}

{\color{red}
	\begin{remm} We observe by classical results that the  SDE (\ref{itoass}) is well-posed since for all $N$ we have $(V^{N}\ast K)\in \mathrm{C}_{b}^{\infty}( \T^d ) $. 
	\end{remm}
}

{\color{red}
	\begin{remm} We observe that when $K$  verifies $K\ast p \in L^{\infty}([0,T],E_{q,r}^{\lambda}) $ for any $p\in C([0,T]; L^{q}(\T^{d}) )$ with $q>d/\lambda $ and $\lambda >0$, it regularises enough to ensure that $K\ast p$ is H\"older continuous in $L^q$, with a H\"older coefficient that determines the convergence rate $\rho$. This is no longer the case when $K$ is a Dirac mass. Nevertheless, assuming initial conditions in $E_{q,r}^{\lambda}$, we are able to prove that $p$ is uniformly bounded in $E_{q,r}^{\lambda}$ and is therefore $(\lambda-d/q)$-H\"older continuous. This is done in Theorem \ref{mainthm}. 
	\end{remm}
}

{\color{red}
	\begin{remm} We note  the convergence (\ref{rate_true}) is deduced from (\ref{rate_moll}), the hypotheses \eqref{initial}
		and the convergence rate for the stochastic convolution in the space $E_{q^\prime,r^\prime}^{-\lambda}$.
	\end{remm}
}

\subsection{Second main result.}

{\color{red}In this section, let $q>d$.} Solutions to \eqref{eq:PDE2} will be understood in the following mild sense:
\begin{defi}\label{def:defMild2}
Given $p_0 \in  H^{\lambda}_{q}(\mathbb{R}^{d})$, $\lambda> \frac{d}{q} $,  $T>0$,
 a function $p$ on $[0,T] \times \R^d$ is said to be a mild solution to \eqref{eq:PDE2} on $[0,T]$ if
\begin{enumerate}
\item $p\in C([0,T]; H^{\lambda}_{q}(\mathbb{R}^{d})) $; 
\item $p$ satisfies the integral equation
\begin{equation}
\label{eq:mildKS2}
p_{t} =  e^{t\Delta} p_0 -  \int_0^t \nabla \cdot ( e^{(t-s)\Delta }   \big( p (\alpha(s,x) +  b(x,p)) \big)\, ds, \quad 0 \leq t \leq T.
\end{equation}
\end{enumerate}
A function $p$ on {\color{red}$[0,\infty) \times \R^d$} is said to be a global mild solution to \eqref{eq:PDE2} if it is a mild solution to \eqref{eq:PDE2} on $[0,T]$ for all $T>0$.
\end{defi}

\begin{enumerate}
    \item[(H1)]\label{H1}  $b:\R^d \times \R_+ \to \R^d$ is a Borel measurable function, continuous and such that there exist two constants $C, L>0$ for which it holds that
\begin{align*}
& |b(x, u)| \leq C, \\
& |b(x, u)-b(y, v)| \leq L(|x-y|+|u-v|) .
\end{align*}
for all $x, y \in \mathbb{R}^d, u,v \in \mathbb{R}_{+}$.

\item[(H2)]\label{H2}  $\alpha \in C([0,T], \mathrm{C}_{b}^{\eta} (\mathbb{R}^{d}; \mathbb{R}^d))$.

\item[(H3)]\label{H3}  $V \in \mathrm{C}_{c}^\infty(\mathbb{R}^d) \cap \mathcal{P}(\mathbb{R}^d)$.
\end{enumerate}
{\color{red}Recall that $p^N_t=V^N \ast S^N_t$, $t \in[0,T]$.}

\begin{thm}\label{mainthm}{\color{red} Let $q\ge 2$ and assume (H1)-(H3). Let $T_{\max}$ be the maximal existence time for \eqref{eq:PDE2}, fix $T\in(0,T_{\max})$ and 
take $\delta>0$ such that $\rho>0$, where
\begin{align*}
	\rho = \min \left\{\frac{\beta}{d}\Big(\eta \wedge \Big(\lambda-\frac{d}{q}\Big)\Big), \, \frac{1}{2} - \beta\Big(1+\frac{\lambda}{d}-\frac{1}{q}\Big)\right\}-\delta. 
\end{align*}
In addition, let the dynamics of the particle system be given by \eqref{SDE} with $\alpha^{i,N}=\alpha$ satisfying (H2) for $i=1,...,N$ and assume
\begin{align}
\lim_{N \to \infty}N^{\rho}\, \| p^N_{0}-p_{0} \|_{\lambda,q}=0, \, \mathbb{P}\,\text{-a.s.} 
\end{align}

Then, we have

\begin{align}
\lim_{N \to \infty}N^{\rho}\sup_{t \in [0,T]} \| p^N_{t}-p_{t} \|_{\lambda,q}=0, \, \mathbb{P}\,\text{-a.s.}
\end{align}}

\end{thm}

{\color{red}
\begin{remm} We observe that the  SDE \eqref{SDE} is well-posed   because  for all $N$ we have    $V^{N}\in C_{b}^{\infty}( \R^d ) $, 
b is Lipschitz in both variables and $\alpha$ is a H\"older continuous function.  
\end{remm}
}

\subsection{Application to Mean Field Games}\label{MFG}

Mean Field Games (MFG) models were first proposed by Lasry and Lions \cite{Larsy} and simultaneously by Huang, Caines, and Malham\'e  \cite{Huang}. 
The field has grown considerably in the last decade, with research taking many different directions, from existence, uniqueness,  regularity and numerical analysis. MFG models aim to describe how populations of agents evolve over time due to their strategic interactions. 
The literature on MFGs is rapidly
growing, and the application of MFG theory is catching on in areas as diverse as
Economics, Biology, and Physics; hence, it is impossible to give
an exhaustive account of the activity on the topic. We refer the reader
to \cite{Carmona}, \cite{Carmona2}, and \cite{Porreta} for a comprehensive presentation of the MFG theory and its applications.

The equation (\ref{SDE}) says that each agent $i$ partially controls its velocity through the strategy vector $\alpha^{i,N}$. This model describes the situation where a single player interacts only with the few people in the surrounding environment, see \cite{Flandoli} and  \cite{Oelschlager87}.

The evolution of the players' positions depends on the strategy profile $\bm{\alpha} ^N=(\alpha^{1,N},...,\alpha^{N,N})$, where each $\alpha^{i,N} \in C_b([0,T], C_b^\eta(\Rd,\Rd))$ and on the initial distribution of states, which we assume is given by $\mu_0^{\otimes N}$, where the law $\mu_0$ has density $p_0$. Each player acts to minimize its own expected cost; more precisely, the player $i$ evaluates the performance of the strategy vector $\bm{\alpha}^N$ according to the cost functional
\begin{equation*}
    J_i^N(\bm{\alpha}^N)=\mathbb{E}\left[ \int_0^T\left(\frac{1}{2}\left|\alpha^{i,N}_s(X^{i,N}_s)\right|^2+f\big(X_s^{i,N}, p_s^N(X_s^{i,N})\big)\right) d s  + g(X^{i,N}_T) \right]
\end{equation*}
\noindent where $(X^{1,N},...,X^{N,N})$ is a solution of (\ref{SDE}) under $\bm{\alpha}^N$. An interesting concept of equilibrium for this optimization problem is that of $\varepsilon$-Nash equilibrium. Given $\varepsilon>0$, a strategy vector $\bm{\alpha}^N$ is called an $\varepsilon$-Nash equilibrium for the $N$-player game if for every $i=1,...,N$,
\[ J^N_i(\bm{\alpha}^N) \le J^N_i([\bm{\alpha}^{N,-i},\beta])+\varepsilon \]
\noindent for all $\beta \in C_b([0,T], C_b^\eta(\Rd,\Rd))$, where $[\bm{\alpha}^{N,-i},\beta]$ indicates the profile obtained from $\bm{\alpha}^N$ by replacing $\alpha^{i,N}$, the strategy of the $i$-th player, by $\beta$.

The existence of an $\epsilon$-Nash equilibrium {\color{red} for the moderate interaction regime} was addressed by Flandoli et al. in \cite{Flandoli}. The authors construct a sequence of approximate Nash equilibria for the $N$-player game by solving the corresponding control problem for one representative player. This can be justified as follows. Fixing one player, say $i$, the influence of the others appears only in the form of $p_s^N(X_s^{i,N}) = \frac{1}{N}\sum_{j=1}^N V^N(X^{i,N}_s-X^{j,N}_s)$. This quantity represents the density of particles in the neighborhood of particle $i$ and is replaced by $p_s(X_s^{i,N})$, the limiting density of particles when $N \to \infty$. In this scenario, the density $p$ is given, and the trajectory of the representative player follows
\[ X_t = X_0 + \int_0^t \big(\alpha(s,X_s)+b(X_s,p(s,X_s))\big)\, ds + W_t\]
with $X_0 \sim p_0 \, dx$. The cost functional is
\[ J(\alpha) = \mathbb{E} \left[ \int_0^T\left(\frac{1}{2}\left|\alpha(s, X_s)\right|^2+f(X_s, p(s,X_s))\right) d s+ g(X_T)\right]\]

\noindent {\color{red} where $f: \Rd \times \R_+ \to \Rd$ satisfies (H1)}. In Section 4 of \cite{Flandoli}, the authors find the optimizing $\alpha$ over the class of bounded continuous functions, under the assumptions (H1) for {\color{red}$b$ and $f$}, (H3) and $g,\partial_i g, p_0 \in C_b(\Rd)$, with

\[ \int_{\Rd} e^{\lambda |x|}\, p_0(x) \, dx < \infty \]
\noindent for all $\lambda>0$. The optimal $\alpha^*(t,x) = -\nabla u(t,x)$, where $(u,p)$ is a mild solution (or equivalently, in this case, a weak solution) of the MFG PDE system
\begin{equation}\label{MFG_PDE}
\Scale[0.94]{
\begin{cases}-\partial_t u-\frac{1}{2} \Delta u-b(x, p(t, x)) \cdot \nabla u+\frac{1}{2}|\nabla u|^2=f(x, p(t, x)), & (t, x) \in[0, T) \times \mathbb{R}^d, \\ \partial_t p-\frac{1}{2} \Delta p+\nabla \cdot [p(t, x)(-\nabla u(t, x)+b(x, p(t, x)))]=0, & (t, x) \in(0, T] \times \mathbb{R}^d, \\ p(0, \cdot)=p_0(\cdot), \quad u(T, \cdot)=g(\cdot), & x \in \mathbb{R}^d. \end{cases}}
\end{equation}

\noindent with $u, \partial_i u, p \in C_b([0,T]\times \Rd)$. The MFG PDE is a coupled system consisting of a backward Hamilton-Jacobi Bellman equation for the value function $u$ and a Kolmogorov forward equation for the density $p$. {\color{red} In \cite{Flandoli} the authors 
 show existence and uniqueness result for the PDE system (\ref{MFG_PDE}). Many authors have studied this type of system in the last years, see for instance 
 \cite{Gomes}, \cite{Larsy}, \cite{Larsy2} \cite{Porreta}, \cite{Porreta2}.}

Furthermore, the authors prove convergence in probability of $S^N_t$ (for fixed $\alpha$) towards $\mu \in C([0,T],\mathcal{P}(\Rd))$, which is absolutely continuous with respect to the Lebesgue measure on $\Rd$, having density $p_t$, the mild solution of the MFG system.

Our contribution in this setting is to prove, as an application of Theorem \ref{mainthm}, the convergence of $p^N$ towards $p$ in the $H^{\lambda}_q(\Rd)$ norm, with explicit rate, when all players use the optimal strategy $\alpha=-\nabla u$. For that, we need to enhance the regularity of $-\nabla u$:

\begin{prop}\label{reg_u}
    Let (H1), (H3), (H4) be in place, and assume also $\nabla g \in C_b^{\eta}(\Rd)$. Then, the value function $u$ coming from the solution pair $(u,p)$ of (\ref{MFG_PDE}) is such that $\nabla u \in C([0,T],C^\eta_b(\Rd;\Rd)) $.
\end{prop}

\begin{proof}
    We have the mild formulation for $u$:
\begin{align*}
    u(t,x) &= e^{(T-t)\Delta}g(x) \\
    &+ \int_t^T e^{(s-t)\Delta} \Big(b(x,p(s,x))\cdot \nabla u(t,x) - \frac{1}{2}|\nabla u(s,x)|^2 + f(x,p(s,x))  \Big) \:\! ds
\end{align*}
$\forall t \in [0,T]$, $\forall x \in \mathbb{R}^d$. Thus we have

\begin{align*}
  \| \nabla u(t,\cdot) \|_{\eta} &=  \| \nabla e^{(T-t)\Delta}g \|_{\eta} \\
    &\quad + \int_t^T \big \| \nabla e^{(s-t)\Delta} \big(b(\cdot,p_s) \cdot \nabla u_s - \frac{1}{2}|\nabla u_s|^2 + f(\cdot,p_s) \big) \big \|_{\eta}\, ds \\
    &\le C \|\nabla g \|_{\eta} \\
    &\quad +  \int_t^T \frac{1}{(s - t)^{\frac{1+\eta}{2}}}\big \| b(\cdot,p_s)\cdot \nabla u_s - \frac{1}{2}|\nabla u_s|^2 + f(\cdot,p_s) \big \|_{\infty}\, ds \\
    &\le C \|\nabla g \|_{\eta} +  C\int_t^T \frac{1}{(s - t)^{\frac{1+\eta}{2}}}\,ds \\
    &= C \|\nabla g \|_{\eta} + C_T < \infty,
\end{align*}
{\color{red}  where, the first inequality is derived using the same computations as in \cite{Flandoli}, Lemma $C.3$, to get, for $h \in L^{\infty}$ and $s>t$,
\begin{align*}
\left\|\nabla e^{(s-t)\Delta}h\right\|_{\eta}\leq \frac{\left\|h\right\|_{\infty}}{(s-t)^{\frac{1+\eta}{2}}}.
	\end{align*}

The second inequality holds since $b,f$ are bounded by hypothesis (H1), and $\nabla u$ is bounded by the definition of solution in \cite{Flandoli}.
}

\end{proof}

{\color{red} Notice that the above result gives conditions for the strategy $\alpha(t,x) = -\nabla u(t,x)$ to satisfy hypothesis (H2)}. Then the following holds, as a corollary of Theorem \ref{mainthm}.

\begin{prop} {\color{red}Let $q\ge 2$, $q>d$ and assume (H1), (H3), and that $\nabla g \in C^\eta_b(\Rd)$.  Let $T_{\max}$ be the maximal existence time for \eqref{eq:PDE2}, fix $T\in(0,T_{\max})$ and 
	take $\delta>0$ such that $\rho>0$, where
\begin{align*}
	\rho = \min \left\{\frac{\beta}{d}\Big(\eta \wedge \Big(\lambda-\frac{d}{q}\Big)\Big), \,  \frac{1}{2} - \beta\Big(1+\frac{\lambda}{d}-\frac{1}{q}\Big)\right\}-\delta. 
\end{align*}
Let the dynamics of the particle system be given by \eqref{SDE} with $\alpha^{i,N}(t,x)=-\nabla u(t,x)$ for $i=1,...,N$, and assume

\begin{align}
\lim_{N \to \infty}N^{\rho} \, \| p^N_{0}-p_{0} \|_{\lambda,q}=0, \, \mathbb{P}\,\text{-a.s.}
\end{align}

Then, we have 

\begin{align}
\lim_{N \to \infty}N^{\rho}\sup_{t \in [0,T]} \| p^N_{t}-p_{t} \|_{\lambda,q}=0, \, \mathbb{P}\,\text{-a.s.}
\end{align}}

\end{prop}

\section{ Proof of Theorem \ref{th:rate2} }\label{subsec:proofRate2}

The proof of the rate of convergence relies on the following mild formulation of the mollified empirical measure:
\begin{equation}
	\label{eq:mildeq}
	\begin{split}
		p^N_t(x)=e^{t\Delta }p^N_0(x) -  \int_0^t \nabla \cdot e^{(t-s)\Delta } \langle S_s^N,  V^N ( x-\cdot)   K\ast p^N_s  \rangle \ ds \\
		- \frac{1}{N} \sum_{i=1}^N \int_0^t  e^{(t-s)\Delta} \nabla V^N (x-X_s^{i,N})\cdot dW^i_s, 
	\end{split}
\end{equation}

\noindent see formula (2.3) in \cite{Pisa}. 

Now we subtract and add the term $\int_{0}^t \nabla\cdot e^{(t-s)\Delta}   \langle S^N_{s},  V^N (x-\cdot)  K\ast p_s(x) \rangle \ ds$, and compare with the mild formulation of $p$. We get
\begin{align*}
	p^N_{t}(x) - p_{t}(x) &= e^{t\Delta} (p^N_{0} - p_{0})(x) \\
	+  \int_0^t  &\nabla \cdot e^{(t-s)\Delta } (p_{s}  (K\ast p_s) - p^{N}_{s}  (K\ast p_s))(x)\, ds  + E_{t}(x) - M^N_{t}(x),
\end{align*}
where we have set
\begin{align}\label{eq:defM}
	E_{t}(x)&:= \int_0^t \nabla \cdot e^{(t-s)\Delta } \langle S^N_{s},  V^N (x-\cdot) \left(K\ast p_s(x)- K\ast p^N_s(\cdot)\right)\rangle \ ds,\nonumber\\
	M^N_{t}(x) &:=\frac{1}{N} \sum_{i=1}^N \int_0^t  e^{(t-s)\Delta} \nabla V^N (x-X_s^{i,N})\cdot dW^i_s. 
\end{align}

Then, we have {\color{red} by $(\ref{Semi})$ with $\lambda = 0$ that,}
\begin{equation}\label{dife}
	\begin{split}
		\| p^N_t-p_{t} \|_{L^{q}(\T^{d})} &\leq \|e^{t\Delta }(p^N_0- p_0 )\|_{L^{q}(\T^{d})} \\
		&\quad + C \int_0^t \frac{1}{\sqrt{t-s}} \| (p_{s}-p_{s}^{N})  K\ast p_s   \|_{L^{q}(\T^{d})} ds\\
		&\quad + \| E_{t}\|_{L^{q}(\T^{d})}  + \|  M^{{N}}_{t}\|_{L^{q}(\T^{d})}. 
	\end{split}
\end{equation}

{\color{red}We will use the bound $\|e^{t\Delta }(p^N_0- p_0 )\|_{L^{q}(\T^{d})}\leq \|p^N_0- p_0\|_{L^{q}(\T^{d})}$,  which holds since the heat semigroup is a contraction semigroup in $L^q$ spaces}. Also we observe 

\[
\|  (p_{s}-p_{s}^{N})  K\ast p_s \|_{L^{q}(\T^{d})} \leq C
\|  p_{s}-p_{s}^{N}   \|_{L^{q}(\T^{d})} . 
\]

Now, we will estimate $ \| E_{t}\|_{L^{q}(\T^{d})} $. 
We subtract and add the term \\ $  \langle S^N_{s},  V^N (x-\cdot)  K\ast p_s(\cdot) \rangle  $, then we have

\[
\| E_{t}\|_{L^{q}(\T^{d})}
\]
\[
\leq \int_{0}^{t}  \frac{1}{\sqrt{t-s}} \|\langle S^N_{s},  V^N (x-\cdot) ( K\ast p_s(x) - K\ast p_s(\cdot)) \rangle   \|_{L^{q}(\T^{d})}
\]
\[
+ \int_{0}^{t}  \frac{1}{\sqrt{t-s}} \|\langle S^N_{s},  V^N (x-\cdot) ( K\ast p_s^{N}(\cdot) - K\ast p_s(\cdot)) \rangle   \|_{L^{q}(\T^{d})}=:I_{1} + I_{2}.
\]

By assumption, $K \ast p \in L^{\infty}([0,T],E_{q,r}^{\lambda})$, thus by Sobolev embedding {\color{red}in (\ref{besov/triebel holder emb})}, $K\ast p_s \in C^{\lambda -\frac{d}{q}}$, and we get
\begin{align*}
	V^N (x-\cdot) |K\ast p_s(x) - K\ast p_s(\cdot)| &\le C V^N (x-\cdot)|x-\cdot|^{\lambda - \frac{d}{q}} \\
	&\le C V^N (x-\cdot)N^{-\frac{\beta}{d}(\lambda - \frac{d}{q})}.
\end{align*}
Therefore, we have 
\begin{align}
	I_{1} &\leq  \int_{0}^{t}  \frac{C}{\sqrt{t-s}} N^{-\frac{\beta}{d}(\lambda-\frac{d}{q})}
	\|  p_{s}^{N} \|_{L^{q}(\T^{d})} \ ds  \nonumber \\
	& \leq C
	\int_{0}^{t}  \frac{1}{\sqrt{t-s}} \|  p_{s}^{N}-  p_{s}\|_{L^{q}(\T^{d})}\  ds  \nonumber \\ 
	&\quad  +  C N^{-\frac{\beta}{d}(\lambda-\frac{d}{q})}  \int_{0}^{t}  \frac{1}{\sqrt{t-s}} \| p_{s}\|_{L^q(\T^{d})} \ ds \nonumber \\
	&= C
	\int_{0}^{t}  \frac{1}{\sqrt{t-s}}  \|  p_{s}^{N}-  p_{s}\|_{L^{q}(\T^{d})}\  ds  +  C_T {N^{-\frac{\beta}{d}(\lambda-\frac{d}{q})}} .\label{E1}
\end{align}

By convolution inequality, we obtain 

\[
I_{2} 
\leq C  \int_{0}^{t}  \frac{1}{\sqrt{t-s}} 
\| p_{s}^{N}-p_{s}\|_{L^{q}(\T^{d})}
\| p_{s}^{N}\|_{L^q(\T^{d})}  ds 
\]

\[
\leq C  \int_{0}^{t}  \frac{1}{\sqrt{t-s}} 
\| p_{s}^{N}-p_{s}\|_{L^{q}(\T^{d})}^{2} \ ds 
\]

\begin{equation}\label{E2}
	+ C  \int_{0}^{t}  \frac{1}{\sqrt{t-s}} 
	\| p_{s}^{N}-p_{s}\|_{L^{q}(\T^{d})}\ ds.
\end{equation}

Now, we will estimate  $\|  M^{{N}}_{t}\|_{L^{q}(\T^{d})}$. {\color{red} If $q \in [1,2) $, we use the estimate $\|  M^{{N}}_{t}\|_{L^{q}(\T^{d})} \lesssim \|  M^{{N}}_{t}\|_{L^{2}(\T^{d})}$. For $q \ge 2$, 
	by Proposition A.8 of \cite{Pisa} and Proposition \ref{borel cantelli}}, we deduce that there exists a random variable $A_0$ with finite moments such that almost surely,
\begin{equation}\label{martingale}
	\sup_{t \in [0,T]}\|  M^{{N}}_{t}\|_{L^{q}(\T^{d})}\leq A_0\, N^{-\frac{1}{2} + \frac{\beta}{2}\big(1 +2 \big(\frac{1}{2}-\frac{1}{q}\big)\vee0\big) + \delta}
\end{equation}

\noindent with $\delta >0$. From  (\ref{dife}),  (\ref{E1}), (\ref{E2}) and (\ref{martingale}) we obtain for all $q \ge 1$

\begin{equation}\label{dife3}
	\begin{split}
		\| p^N_t-p_{t} \|_{L^{q}(\T^{d})} &\leq \|p^N_0- p_0 \|_{L^{q}(\T^{d})} \\
		&\quad + C \int_0^t \frac{1}{\sqrt{t-s}} \| p_{s}-p_{s}^{N}  \|_{L^{q}(\T^{d})}^{2} \ ds\\
		&\quad +  C \int_0^t \frac{1}{\sqrt{t-s}} \| p_{s}-p_{s}^{N}  \|_{L^{q}(\T^{d})} \ ds\\
		&\quad +  C_{T}  N^{-\frac{\beta}{d}\big(\lambda-\frac{d}{q}\big)}  +  A_0\:\! N^{-\frac{1}{2} + \frac{\beta}{2}\big(1 +2 \big(\frac{1}{2}-\frac{1}{q}\big)\vee0\big) + \delta}. 
	\end{split}
\end{equation}

We set

\begin{equation}
	T^{N}=\inf \left\{t \geq 0: \| p_{t}-p_{t}^{N}  \|_{L^{q}(\T^{d})} \geq N^{-\rho}  \right\} \wedge T. 
\end{equation}

Then we have

\begin{equation}\label{dife_3}
	\begin{split}
		\| p^N_{t \wedge T^{N}}-p_{t \wedge T^{N}} \|_{L^{q}(\T^{d})} &\leq \|p^N_0- p_0 \|_{L^{q}(\T^{d})} \\
		&\quad + C \int_0^{t \wedge T^{N}} \frac{1}{\sqrt{t-s}} \| p_{s}-p_{s}^{N}  \|_{L^{q}(\T^{d})} \ ds+  C_{T}  N^{-2\rho}\\
		&\quad   +  C_{T}  N^{-\frac{\beta}{d}\big(\lambda-\frac{d}{q}\big)}  +  A_0\:\! N^{-\frac{1}{2} + \frac{\beta}{2}\big(1 +2 \big(\frac{1}{2}-\frac{1}{q}\big)\vee0\big) + \delta}.  
	\end{split}
\end{equation}

By Gr\"onwall Lemma

\begin{equation}\label{dife4}
	\begin{split}
		\sup_{t \in [0,T]} \| p^N_{t \wedge T^{N}}-p_{t \wedge T^{N}} \|_{L^{q}(\T^{d})}
		&\leq  C_{T}  \|p^N_0- p_0 \|_{L^{q}(\T^{d})} +  C_{T}  N^{-2\rho} \\
		&\quad +C_{T}  N^{-\frac{\beta}{d}\big(\lambda-\frac{d}{q}\big)}  + A_0\:\! N^{-\frac{1}{2} + \frac{\beta}{2}\big(1 +2 \big(\frac{1}{2}-\frac{1}{q}\big)\vee0\big) + \delta}. 
	\end{split}
\end{equation}

Then 

\begin{equation}\label{conve}
	\lim_{N \to \infty} N^{\rho}\sup_{t \in [0,T]} \| p^N_{t \wedge T^{N}}-p_{t \wedge T^{N}} \|_{L^{q}(\T^{d})}=0
\end{equation}

Suppose there are infinite numbers $N_{k}$ such that $T^{N_{k}}< T$. From (\ref{conve}) we have 

\[
\lim_{k \rightarrow\infty} N_{k}^{\rho} \sup_{t \in [0, T^{N_k}]} \| p^{N_{k}}_{t}-p_{t} \|_{L^{q}(\T^{d})} 
=0. 
\]

Therefore  for $k$ large  and   $t\leq T^{N_{k}}$ we obtain

$$
N_{k}^{\rho} \  \| p^{N_{k}}_t-p_t \|_{L^{q}(\T^{d})}   <  \epsilon,
$$

\noindent this  contradicts the definition of $T^{N_{k}}$.  So we conclude that $T^{N}=T$
for $N$ large enough. Thus we have

\begin{equation}\label{conve2}
	{\color{red}\lim_{N \to \infty} N^{\rho}\sup_{t \in [0,T]} \| p^N_{t}-p_{t} \|_{L^{q}(\T^{d})}=0, \, \mathbb{P}\,\text{-a.s}.}
\end{equation}

Now, we will show the second statement of our theorem. First, we observe that 

\begin{align*}
	\! \langle S^N_{t} \!\!\:- p_{t}, \phi \rangle &= 
	\langle S^N_{0} - p_{0}, e^{t\Delta}\phi \rangle \!\!\;+\!\!\;  \!\!\:\int_0^t  \!\!\: \langle S^{N}_{s}  (K\ast p_s^{N})- p_{s}  (K\ast p_s),  \nabla \!\! \; \cdot \!\!\; e^{(t-s)\Delta } \phi \rangle \!\: ds \\
	&\quad +  \langle M^N_{t},\phi \rangle, 
\end{align*}

\noindent where $\phi \in E_{q,r}^{\lambda}$ and
\begin{align*}
	\langle M^N_{t}, \phi \rangle &=\frac{1}{N} \sum_{i=1}^N \int_0^t  \nabla e^{(t-s)\Delta}  \phi(X_s^{i,N})\cdot dW^i_s. 
\end{align*}

Now, we subtract and add the term 
$\langle   S^{N}_{s},   (K\ast p_s) \, \nabla \cdot e^{(t-s)\Delta } \phi \rangle$ to have

\begin{align*}
	\langle S^N_{t} - p_{t}, \phi \rangle &= 
	\langle S^N_{0} - p_{0}, e^{t\Delta}\phi \rangle+  \int_0^t  \langle  S^{N}_{s} - p_{s} ,  (K\ast p_s) \, \nabla \cdot e^{(t-s)\Delta } \phi \rangle ~ ds \\
	&\quad  +  \int_0^t  \langle S^{N}_{s} ,  (K\ast p^N_s -K\ast p_s)\nabla \cdot e^{(t-s)\Delta } \phi \rangle ~ ds   +  \langle M^N_{t},\phi \rangle.
\end{align*}

{\color{red}Note that $\lambda>\frac{d}{q}$ implies  $E^{\lambda}_{q,r}$ is an algebra. Then by inequality (\ref{grad heat besov-triebel}) }

\begin{equation}\label{dife2}
	\begin{split}
		\| S^N_t-p_{t} \|_{E_{q^\prime,r^\prime}^{-\lambda}} &\leq \|e^{t\Delta }(S^N_0- p_0 )\|_{E_{q^\prime,r^\prime}^{-\lambda}} \\
		&\quad + C \int_0^t \frac{1}{\sqrt{t-s}} \| K\ast p\,\|_{L^{\infty}([0,T], E_{q,r}^{\lambda})} \, \| S_{s}^{N}-p_{s} \|_{_{E_{q^\prime,r^\prime}^{-\lambda}}} ds\\
		&\quad +  
		\int_0^t  \sup_{\| \phi\|_{E_{q,r}^{\lambda}}\leq 1}  \langle S^{N}_{s} ,  (K\ast p^N_s -K\ast p_s)\nabla \cdot e^{(t-s)\Delta } \phi \rangle \, ds \\
		&\quad + \|  M^{{N}}_{t}\|_{_{E_{q^\prime,r^\prime}^{-\lambda}}}. 
	\end{split}
\end{equation}

We observe that 

\[
\|e^{t\Delta }(S^N_0- p_0 )\|_{E_{q^\prime,r^\prime}^{-\lambda}}\leq \|S^N_0- p_0 \|_{E_{q^\prime,r^\prime}^{-\lambda}}.
\]

{\color{red}Also since $\lambda >\frac{d}{q}$, by inequalities (\ref{besov/triebel holder emb}) and (\ref{grad heat besov-triebel}) we have}
\[
\int_0^t  \sup_{\| \phi\|_{E_{q,r}^{\lambda}}\leq 1}  \langle S^{N}_{s} ,  (K\ast p^N_s -K\ast p_s)\nabla \cdot e^{(t-s)\Delta } \phi \rangle    \ ds
\]

\[
\leq \int_0^t   \sup_{\| \phi\|_{E_{q,r}^{\lambda}}\leq 1}  \| K\ast p^N_s -K\ast p_s \|_{\infty} 
\|\nabla \cdot e^{(t-s)\Delta } \phi\|_{\infty}  \ ds 
\]
{\color{red}\[
	\leq \int_0^t   \sup_{\| \phi\|_{E_{q,r}^{\lambda}}\leq 1}  \| K\ast p^N_s -K\ast p_s \|_{\infty} 
	\|\nabla \cdot e^{(t-s)\Delta } \phi\|_{E^{\lambda}_{q,r}}  \ ds 
	\]}
\[
\le \int_0^t \frac{1}{\sqrt{t-s}}   \| K\ast p_s^N -K\ast p_s \|_{\infty}  \ ds. 
\]

\begin{equation}\label{ine}
	\le \int_0^t \frac{1}{\sqrt{t-s}}  \,  \| p^N_s - p_s \|_{L^q(\mathbb{T}^d)}  \ ds. 
\end{equation}

From the first   statement of our theorem, we obtain 
\begin{align}
	{\color{red}\lim_{N \to \infty}N^{\rho}\sup_{t \in [0,T]} \| p^N_{t}-p_{t} \|_{L^{q}(\T^{d})}=0, \, \mathbb{P}\text{- a.s.} }
\end{align}

By Sobolev embedding {\color{red} in (\ref{besov/triebel bessel emb}),}
\[ \|  M^{{N}}_{t}\|_{{E_{q^\prime,r^\prime}^{-\lambda}}}\leq C \|  M^{{N}}_{t}\|_{H^{-\lambda {\color{red}+\delta}}_{q^\prime}} .\]

Now, by the proof of Proposition A.12 in \cite{Pisa} {\color{red} and Proposition \ref{borel cantelli}}, we deduce that {\color{red} if $q^\prime \ge 2$}, there exists a random variable $\Tilde{A}_0$ with finite moments such that almost surely,

\begin{equation}\label{Mart}
	\|  M^{{N}}_{t}\|_{H^{-\lambda{\color{red}+\delta}}_{q^\prime}}\leq
	\Tilde{A}_0 \:\! N^{-\frac{1}{2} + \frac{\beta}{2}\big(1 -\frac{2\lambda}{d} +2 \big(\frac{1}{q}-\frac{1}{2}\big)\vee0\big)+\delta}
\end{equation}

\noindent where we used $\frac{1}{2}-\frac{1}{q^\prime}=\frac{1}{q}-\frac{1}{2}$. {\color{red} Now, if $q^\prime \in [1,2)$, we observe that 
	\begin{align*}\|  M^{{N}}_{t}\|_{H^{-\lambda{\color{red}+\delta}}_{q^\prime}} &= \|(I-\Delta)^{\frac{-\lambda + \delta}{2}}  M^{{N}}_{t}\|_{L^{q^\prime}(\mathbb{T}^d)} \\ &\lesssim  \|(I-\Delta)^{\frac{-\lambda + \delta}{2}}  M^{{N}}_{t}\|_{L^{2}(\mathbb{T}^d)}  = \|  M^{{N}}_{t}\|_{H^{-\lambda{\color{red}+\delta}}_{2}}. 
	\end{align*} 
	Hence the estimate (\ref{Mart}) is valid for all $q^\prime \ge 1$.} Then, by Gr\"onwall lemma and estimates (\ref{dife2})-(\ref{Mart}) we deduce the second statement, {\color{red} with the same rate $\rho$. In fact, one can verify that 
	\[ \frac{1}{2} - \frac{\beta}{2}\Big(1 +2 \Big(\frac{1}{2}-\frac{1}{q}\Big)\vee0\Big) \le \frac{1}{2} - \frac{\beta}{2}\Big(1 -\frac{2\lambda}{d} +2 \Big(\frac{1}{q}-\frac{1}{2}\Big)\vee0\Big)\]
	\noindent by checking for $q\le 2$ and $q>2$, and using the hypothesis $q> \frac{d}{\lambda}$, which implies $\frac{\lambda}{d}>\frac{1}{q}$. }

\section{ Proof of Theorem \ref{mainthm} }\label{subsec:proofRate1}

From the mild formulations of $p_t$ in (\ref{eq:mildKS}) and of $p^N_t$, similar to (\ref{eq:mildeq}), we have

\begin{align}
&\|  p_{t}^N-p_{0}\|_{\lambda,q} \nonumber \\
&\quad \le  \|e^{t\Delta } \big( p_{0}^N - p_{0}\big)\|_{\lambda,q} \nonumber \\
& \quad \quad + \int_0^t \| \nabla e^{(t-s)\Delta }  \big[\big(V^N \ast \big(\alpha_{s}(\cdot))S^N_s\big)\big)-\alpha_{s} p_{s} \big]\|_{\lambda,q}\, ds \nonumber \\
	& \quad \quad + \int_0^t \| \nabla e^{(t-s)\Delta } \big[\big(V^N \ast \big(b(\cdot, p_{s}^N( \cdot ))S^N_s\big)\big)- b(\cdot,p_{s})p_{s} \big] \|_{\lambda,q} \, ds \nonumber \\ 
    & \quad \quad + \|{M^N_t} \|_{\lambda,q} \nonumber \\
	& \quad \le C\|p^N_0 - p_0\|_{\lambda,q} + (I) + (II) + \|M^N_t\|_{\lambda,q}. \label{eq4}
\end{align}
Now, we deal with the above-defined terms separately. 

\smallskip

We deduce {\color{red}the following from inequality $(\ref{Semi})$:}

\[ 
(I) \le  \int_0^t \frac{1}{(t-s)^{(\lambda+1)/2}}   \, 
\|\big[\big(V^N \ast \big(\alpha_{s}(\cdot)S^N_s\big)\big)-\alpha_{s}p_{s} \big]\|_{q} \,.
\]

We observe that
{\color{red} \begin{align}
[V^N \ast \big(\alpha_{s} S^N_s\big)-\alpha_{s}p_{s}^N ](x) = \left\langle S_s^N, V^N(x- \cdot) (\alpha_{s}( \cdot) - \alpha_{s}(x))\right\rangle \label{triang holder alpha}.
 \end{align} }

Then 
\begin{align*}
 \|\big( V^N \ast \big(\alpha_{s} S^N_s\big)\big)-\alpha_{s} p_{s} \|_{q} & 
  \le \, \| \big( V^N \ast \big(\alpha_{s} S^N_s\big)\big)-\alpha_{s} p_{s}^N \|_{q} \,  \\
 & \qquad +   \, \|\alpha_{s}p_{s}^N - \alpha_{s}p_{s} \|_{q} \,  \\
 & =: \big [(I.1) + (I.2)\big ]\, .
\end{align*}

{\color{red}Using the Hölder regularity of \(\alpha\) and the assumption \(\text{supp }V \subset \{y \in \mathbb{R}^d \mid |y| < 1\}\), as well as \(V^N = N^{\beta}V\big(N^{\frac{\beta}{d}} \cdot \big)\), \((\ref{triang holder alpha})\) implies
\begin{align*}
(I.1)  \le  C_\alpha N^{-\frac{\beta}{d} \eta} \,  \|p_{s}^N \|_{q} \le    C_\alpha N^{-\frac{\beta}{d} \eta} \|p_{s}\|_{q}+ \,  C_\alpha \|p_{s}^N -p_{s}\|_{\lambda, q } . \nonumber
\end{align*}}

We also have 
\begin{align*}
 (I.2) &=  \|\alpha_{s}p_{s}^N - \alpha_{s}p_{s} \|_{q} \le
 \|\alpha\|_{\infty}\|p_{s}^N - p_{s} \|_{\lambda,q}.
\end{align*}

From the above estimations, we obtain

\begin{equation}
  (I) \le C N^{-\frac{\beta}{d} \eta} \int_0^t \frac{1}{(t-s)^{(\lambda +1)/2}}   \, \|p_{s}\|_{q} \, ds +
		C\int_0^t \frac{1}{(t-s)^{(\lambda +1)/2}} \, \|p^N_s -p_s\|_{\lambda,q} \, ds.\label{lemma1}
\end{equation}

From (\ref{Semi}) we get

\[
(II) \leq  \int_0^t   \frac{1}{(t-s)^{(\lambda+1)/2}} \| V^N \ast \big(b(\cdot, p_{s}^N)S^N_s\big)- b(\cdot,p_{s})p_{s}  \|_{q} \, ds .
\]

Now we subtract and add the term   $V^N \ast \big(b(\cdot, p_{s}( \cdot ))S^N_s\big)$ to get

\[
(II)
\]
\[
 \leq  \int_0^t 
\frac{1}{(t-s)^{(\lambda+1)/2}} 
\| \big(V^N \ast \big(b(\cdot, p_{s}^N( \cdot ))S^N_s\big)\big)- \big(V^N \ast \big(b(\cdot, p_{s}( \cdot ))S^N_s\big)\big)  \|_{q}
  \, ds
	\]
\[		
+\int_0^t  \frac{1}{(t-s)^{(\lambda+1)/2}}   \| \big(V^N \ast \big(b(\cdot, p_{s}( \cdot ))S^N_s\big)\big)  - b(\cdot,p_{s})p_{s}  \|_{q}
  \, ds.
\]

\[
  =:  \big [(II.1) + (II.2)\big ] \, .
\]

\quad  From hypothesis (H1)  we obtain 

\[
 (II.1) 
\]
\[
\leq  \int_0^t 
\frac{1}{(t-s)^{(\lambda+1)/2}}  \| p_{s}^N- p_{s} \|_{\infty} \
\| p_{s}^{N}  \|_{q}
  \, ds
\]

\[
\leq  \int_0^t 
\frac{1}{(t-s)^{(\lambda+1)/2}}  \| p_{s}^N- p_{s} \|_{\lambda,q}^{2} \
  \, ds  +  \int_0^t 
\frac{1}{(t-s)^{(\lambda+1)/2}}  \| p_{s}\|_{q} \  \| p_{s}^N- p_{s} \|_{\lambda,q} \
  \, ds.
\]

For the other one, we subtract and add the term   $b(x,p_{s}(x)) p_{s}^{N}(x)$ we deduce

\[
 (II.2)
\]		
\[		
\le \int_0^t  \frac{1}{(t-s)^{(\lambda+1)/2}}  \| V^N \ast \big(b(\cdot, p_{s}(\cdot))-b(x, p_{s}(x))\big)S^N_s \|_{q}
  \, ds
\]
\[
+ \int_0^t  \frac{1}{(t-s)^{(\lambda+1)/2}}   \|b(\cdot,p_{s}) p_{s}^{N}  - b(\cdot,p_{s})p_{s}  \|_{q}
  \, ds
\]
\[
=: [(II.2.1)+ (II.2.2)].
\]
By Sobolev embedding {\color{red} in (\ref{bessel holder emb})} and hypothesis (H1), we arrive at

\[
 (II.2.1)
\]
\[
\leq C \big(N^{-\frac{\beta}{d}}+N^{-\frac{\beta}{d}(\lambda-d/q)}\big) \int_0^t  \frac{1}{(t-s)^{(\lambda+1)/2}} \|p_{s}^{N} \|_{q}   \, ds
\]
\[
\leq  C \int_0^t  \frac{1}{(t-s)^{(\lambda+1)/2}} \|p_{s}^{N}-p_{s}\|_{\lambda,q}   \, ds
\]
\[
+ C \big(N^{-\frac{\beta}{d}}+N^{-\frac{\beta}{d}(\lambda-d/q)}\big) \int_0^t  \frac{1}{(t-s)^{(\lambda+1)/2}} \|p_{s} \|_{q}   \, ds.
\]

Also, we have

\[
(II.2.2)
\]
\[
\leq
\int_0^t  \frac{C}{(t-s)^{(\lambda+1)/2}}   \| p_{s}^{N}  - p_{s}  \|_{\lambda,q}  \, ds.
\]

{\color{red}Since $q\ge2$,} by the proof of Proposition A.12 of \cite{Pisa}, {\color{red}and  Proposition \ref{borel cantelli}}, we deduce that there exists a random variable $A^\prime_0$ with finite moments such that almost surely,
\[
\sup_{t \in [0,T]}\|M^N_t\|_{\lambda,q}
\leq A^\prime_0 \:\!    N^{-\frac{1}{2} + \frac{\beta}{2}\big(1+\frac{2\lambda}{d} +2 \big(\frac{1}{2}-\frac{1}{q}\big)\big)+\delta}. 
\]

From the previous estimations, we get

\begin{equation}\label{dife_4}
\begin{split}
\| p^N_t-p_{t} \|_{\lambda,q} &\leq \|p^N_0- p_0 \|_{\lambda,q} \\
&\quad + C \int_0^t \frac{1}{(t-s)^{(1+\lambda)/2}} \| p_{s}-p_{s}^{N}  \|_{\lambda,q} \ ds\\
&\quad +  C \int_0^t   \frac{1}{(t-s)^{(1+\lambda)/2}} \| p_{s}-p_{s}^{N}  \|_{\lambda,q}^{2} \ ds\\
&\quad  +  C_{T}  N^{-\frac{\beta}{d}\eta }  + C_{T}  N^{-\frac{\beta}{d}\big(\lambda-\frac{d}{q}\big)}  +  A^\prime_0 \:\! {\color{red}N^{-\frac{1}{2} + \beta\big(1+\frac{\lambda}{d} -\frac{1}{q}\big)+\delta}}.
\end{split}
\end{equation}

	We set

\begin{equation}
    T^{N}=\inf \left\{t \geq 0: \| p_{s}-p_{s}^{N}  \|_{\lambda,q} \geq N^{-\rho}  \right\} \wedge T. 
\end{equation} 

Following line by line the proof of the Theorem \ref{th:rate2} we conclude the proof.

\newpage
\appendix

\section{Appendix}

  {\color{red}
We have the following local well-posedness for the  Fokker-Planck equation  (\ref{eq:PDE}). The proof of the following result  is based on 
	Lemma A.2 in \cite{ORT}.

\begin{prop}\label{propPDE}
Assume that $u_{0}\in L^q(\T^d)$ and $K \in L^{q^\prime}(\mathbb{T}^d)$, with $1/q + 1/{q^\prime}=1$. Then there exists a maximal time $T^*\in (0,+\infty]$ and a unique mild solution to the Fokker-Planck equation 
which is in $C([0,T],L^q(\T^d))$, for any $T<T^*$.
\end{prop}

\begin{proof}

Denote $\mathcal{X} = \mathcal{C}\left([0,T],L^q(\T^d)\right)$ and $ \|\cdot\|_{\mathcal{X}}$ the associated norm. Then define $B(u,v)_t= \int_0^t \nabla \cdot e^{(t-s)\Delta }  (u_{s} K\ast v_s)\, ds$ for $u,v \in \mathcal{X}$ and $t \in [0,T]$. In view of \eqref{Semi}, we have 
\begin{align*}
\|B(u,v) \|_{\mathcal{X}} &\leq C  \int_0^T \frac{\|u_s K\ast v_s\|_{L^q(\mathbb{T}^d)}}{\sqrt{T-s}}\,  ds\\
&\leq C  \int_0^T \frac{\|u_s\|_{L^q(\mathbb{T}^d)} \|K\ast v_s\|_{L^\infty(\T^d)} }{\sqrt{T-s}}\,  ds .
\end{align*}

 By convolution inequality we deduce that 

\[
\|B(u,v) \|_{\mathcal{X}} \leq C \, T^{1/2}   \|u\|_{\mathcal{X}} \|v\|_{\mathcal{X}}  . 
\]

Then by the contraction principle, see for instance Theorem 13.2 in \cite{Lemari}, we deduce the result for $T$ small enough.

\end{proof}

}

\section{Appendix}

\quad In this Appendix, we recall some results that we used throughout the text. The following is a generalization of the classical Gr\"onwall's Lemma (Lemma 7.1.1 in \cite{Henry}):

\begin{prop}\label{Gronwall}
    Suppose $b \ge0, \chi \ge 0$ and the functions $a(t), u(t)$ are nonnegative and locally integrable on $0\le t <T$, some $T <\infty$ with
    \[ u(t) \le a(t) + b \int_0^t (t-s)^{\chi- 1}u(s) \, ds  \]
    \noindent on this interval. Then
    \[ u(t) \le a(t) + \theta \int_0^t E_\chi^\prime(\theta(t-s))a(s)\, ds, \quad 0\le t < T, \]
    \noindent where 
    \[ \theta = (b\, \Gamma (\chi))^{1/\chi}, \, E_\chi(z)=\sum_{n=0}^\infty \frac{z^{n\chi}}{\Gamma(n\chi+1)}, \text{ and } E_\chi^\prime(z) = \frac{d}{dz}E_\chi(z). \]
    In particular, for $\chi=\frac{1}{2}$, we have the nice estimation $e^t \le E_{\frac{1}{2}}(t) \le 2e^t$.
\end{prop}
\begin{remm}
     A typical application of the above result is to obtain estimates for $\sup_{t \in [0,T]} u(t)$, as follows:
\begin{align*}
    u(t) &\le \sup_{t \in [0,T]} a(t) + \theta \int_0^t E_\chi^\prime(\theta(t-s))a(s)\, ds \\
    &\le \Big(\sup_{t \in [0,T]} a(t)\Big)(1+\int_0^t E_\chi^\prime(\theta(t-s))\, ds \\
    \Rightarrow \sup_{t \in [0,T]} u(t) &\le \Big(\sup_{t \in [0,T]} a(t)\Big)(1+E_\chi(\theta T)).
\end{align*}

\end{remm}
{\color{red}
	The next result is an application of the classical Borel-Cantelli's Lemma (Lemma 2.1 in \cite{Kloeden Neuenkirch}).  
\begin{prop} \label{borel cantelli}
Let $\rho > 0$ and $C(m) \in [0,\infty)$ for $m \geq1$. In addition, let $Z_N$, $N \in \mathbb{N}$, be a sequence of random variables such that
\begin{align*}
\left(\mathbb{E}\left|Z_N\right|^m\right)^{\frac{1}{m}} \leq C(m) N^{-\rho},
\end{align*}
for all $m \geq1$ and $N \in \mathbb{N}$. Then for all $\delta >0$, there exists a random variable $A_{\delta}$ such that almost surely
\begin{align*}
\left|Z_N\right| \leq A_{\delta}N^{-\rho + \delta}.
\end{align*}
Moreover, 
\begin{align*}
\mathbb{E}\left|A_{\delta}\right|^m < \infty,
\end{align*}
for all $m \geq1$.
\end{prop}	
	}

\medskip

\section*{Declarations}

Conflict of Interest Data Sharing not applicable to this article as no datasets were generated or analysed during the current study. 

\section*{Acknowledgements}

Author Christian Olivera is partially supported by FAPESP by the grant  2020/04426-6,  by FAPESP-ANR by the grant Stochastic and Deterministic Analysis for Irregular Models 2022/03379-0 and  CNPq by the grant 422145/2023-8. 
Josu\'e Knorst is partially supported by FAPESP by the grant FAPESP 2022/13413-0 and Alexandre B. de Souza  was financed in part by the Coordenação de Aperfeiçoamento de Pessoal de Nível Superior - Brasil (CAPES) - Finance Code 001.

\bibliographystyle{siam}

\end{document}